\documentclass{article}
\usepackage{tikz-cd} 
\usepackage{tikz}
\usepackage{relsize}
\usepackage{cite}
\usepackage{color}
\usepackage{amsmath,amsthm,amssymb}
\usepackage{enumitem}
\usepackage{algorithm}
\usepackage{algpseudocode}
\usepackage[margin=2.6cm]{geometry}
\usepackage{mathdots}

\makeatletter
\usepackage{comment}
\let\wfs@comment@comment\comment
\let\comment\@undefined

\usepackage{changes}
\let\wfs@changes@comment\comment
\let\comment\@undefined

\newcommand\comment{%
    \ifthenelse{\equal{\@currenvir}{comment}}
    {\wfs@comment@comment}
    {\wfs@changes@comment}%
}
\makeatletter
\def\namedlabel#1#2{\begingroup
    #2%
    \def\@currentlabel{#2}%
    \phantomsection\label{#1}\endgroup
}
\makeatother

\usepackage{scalerel,stackengine}
\stackMath
\usepackage{amsmath,amssymb,amsthm,mathrsfs,amsfonts,dsfont,stmaryrd} 
\usepackage{multirow}
\usepackage{amscd}   
\usepackage{fullpage}
\usepackage[all]{xy}  
\usepackage{mathrsfs,enumitem}
\usepackage[titletoc,toc,title]{appendix}
\usepackage{xcolor}  
\usepackage[normalem]{ulem}
\usepackage{diagbox}
\usepackage{caption}
\usepackage[colorlinks = true,
            linkcolor = black,
            urlcolor  = blue,
            citecolor = red,
            anchorcolor = blue]{hyperref}
\usepackage{graphicx}
 \usepackage{color}
 \usepackage{tikz, framed}
\usetikzlibrary{arrows,shapes,matrix,decorations.pathmorphing}
\usetikzlibrary{backgrounds}
\usepackage[english]{babel} 

\theoremstyle{definition}
\newtheorem{theorem}{Theorem}[section]
\newtheorem{definition}[theorem]{{{Definition}}}
\newtheorem{example}[theorem]{{{Example}}}
\newtheorem{notation}[theorem]{{{Notation}}}

\newtheorem{conjecture}[theorem]{{{Conjecture}}}
\newtheorem{corollary}[theorem]{{{Corollary}}}

\newtheorem{lemma}[theorem]{{{Lemma}}}

\newcommand{\sq}{\mathbin{\square}}

\title{A cyclic flat embedding theorem for transversal $q$-matroids}
\author{Andrew Fulcher}
\date{}

\begin{document}

\maketitle

\begin{abstract}
Cyclic flats form a common structural invariant of both matroids and $q$-matroids, determining these objects through their weighted lattices of cyclic flats. In this paper we exploit this perspective to establish a correspondence between matroids and a subclass of $q$-matroids that we call coordinate $q$-matroids. 

Our main result is a cyclic flat embedding theorem showing that the cyclic flat structure of a transversal matroid is preserved under this correspondence. This provides a mechanism for transferring structural properties from matroid theory to the $q$-matroid setting. As an application, we show that nested $q$-matroids are transversal and therefore representable.

Finally, we illustrate the usefulness of this perspective by analysing transversal $q$-matroids under binary operations. We prove that the class of transversal $q$-matroids is closed under the free product and propose a natural presentation for the direct sum motivated by the corresponding construction for matroids.
\end{abstract}

\section{Introduction}

Matroid theory has a rich and well-developed history, dating back to Whitney's foundational work in 1935 \cite{Whitney1935}. The theory provides an axiomatic framework that captures a wide range of combinatorial structures \cite{Welsh1976}. One central example is the association of matroids to error-correcting codes. Matroids arising from linear codes are called representable matroids and have been the subject of extensive study. The theory of representable matroids has seen significant development, as well as applications in coding theory \cite{BritzShiromoto2016,westerback_on_comb_LRCs}.

A prevailing theme in combinatorics is that of $q$-analogues, where combinatorics over a set is extended to that over a vector space. The theory of $q$-matroids first appeared in Crapo's PhD thesis, but was formalised only much later in \cite{JurriusPellikaan2018}, where several concepts from matroid theory were extended to the vector space setting. Since then, many further notions from matroid theory have been shown to extend, often in non-trivial ways, to the $q$-analogue \cite{ByrneCeriaJurrius2022}. However, other concepts remain elusive or have been shown not to extend at all.

One particularly challenging notion is that of representability. Compared to the classical setting, the theory of representable $q$-matroids poses additional difficulties and remains comparatively undeveloped. At the same time, there is strong motivation for a deeper understanding of representable $q$-matroids, since this class corresponds precisely to those $q$-matroids arising from vector rank-metric codes.

A foundational result in matroid theory is that the class of representable matroids is closed under the binary operations of direct sum and free product \cite{CrapoSchmitt2005,Oxley2011}. Analogous operations have been developed for $q$-matroids \cite{CeriaJurrius2024DirectSum, AlfaranoByrneFulcher2025}. However, the behaviour of representability under these operations is markedly different in the $q$-analogue. In particular, it has been shown that the direct sum of representable $q$-matroids need not be representable \cite{GluesingLuerssenJany2025}. Moreover, it remains open whether the free product of representable $q$-matroids is representable.

Understanding when representability of $q$-matroids is preserved under binary operations is therefore an important problem. Such a characterisation would provide structural tools for constructing new vector rank-metric codes. Some progress in this direction has recently been made in \cite{alfarano2026representability}, which also highlights connections between $q$-matroid theory and finite geometry.

One of the rare instances in which a concept from matroid theory extends to $q$-matroid theory in a particularly direct way is that of \emph{cyclic flats}. This has been explored in \cite{AlfaranoByrne2024CyclicFlats,AlfaranoByrneFulcher2025,GluesingLuerssenJany2024}, where the theory of cyclic flats for $q$-matroids was developed and their behaviour under binary operations was studied.

In this paper we formalise a correspondence between matroids and a subclass of $q$-matroids that we call \emph{coordinate $q$-matroids}, via their weighted lattices of cyclic flats. Through this correspondence we prove a cyclic flat embedding theorem showing that transversal matroids embed into transversal $q$-matroids. This embedding is the central contribution of the paper and provides a mechanism for transferring structural results from transversal matroid theory to the $q$-matroid setting. As an immediate consequence, we obtain the representability of a large class of $q$-matroids not previously known to be representable, including the fundamental class of \emph{nested} $q$-matroids.

This cyclic flat embedding theorem suggests a canonical method of embedding matroids into $q$-matroids in such a way that certain structural properties are preserved. In particular, it raises the question of whether cyclic flat embedding provides a canonical means of embedding representable matroids into representable $q$-matroids. This is only one of several natural questions that arise from this perspective.

We also investigate the behaviour of transversal $q$-matroids under binary operations. We provide a constructive proof that the class of transversal $q$-matroids is closed under the free product. Moreover, using the cyclic flat embedding theorem, we obtain the analogous result for the class of coordinate transversal $q$-matroids. In doing so, we propose a presentation for the direct sum of transversal $q$-matroids motivated by the corresponding presentation in the matroid setting.

More broadly, the cyclic flat embedding perspective suggests a guiding principle for extending constructions from matroid theory to the $q$-analogue. If a matroid operation or property admits a meaningful $q$-analogue, then one may expect that the cyclic flat embedding of a matroid possessing this property should yield a $q$-matroid exhibiting the corresponding $q$-analogue. In this sense, cyclic flat embedding provides a natural testing ground for proposed extensions of matroid-theoretic constructions to the $q$-matroid setting.

The structure of the paper is as follows. Section~2 establishes notation and definitions, and recalls several results that will be required later in the paper. Section~3 contains the main result of this paper. In this section, we formalise a correspondence between matroids and coordinate $q$-matroids and show that the transversality (and thus representability) of a coordinate $q$-matroid can be decided by the transversality of its corresponding matroid. As a corollary, we show that all nested $q$-matroids are transversal and representable. 
Section~4 establishes the closure of the class of transversal $q$-matroids under free products. We also show that closure of the class of coordinate transversal $q$-matroids under direct sums follows from the cyclic flat embedding theorem of the previous section. This yields a further class of representable $q$-matroids.

\section{Preliminaries}

We assume that the reader has some familiarity with matroid theory \cite{Welsh1976} and $q$-matroid theory \cite{JurriusPellikaan2018}, for which we have provided references. $q$-matroids can be defined in several cryptomorphic ways. In this paper, we focus on definitions via the \emph{rank function}, \emph{independent spaces}, and \emph{cyclic flats}. Our primary viewpoint will be through the rank function.

\begin{notation}
    We denote by $M=(E,\rho)$ a $q$-matroid on the ambient space $E$ with rank function $\rho$. We denote the set of independent spaces (those $X\leq E$ such that $\rho(X)=\dim(X)$) of $M$ by $\mathcal{I}(M)$.
\end{notation}

The theory of cyclic flats for $q$-matroids was formalised in \cite{AlfaranoByrne2024CyclicFlats}. We recall the definition here and establish notation.

\begin{definition}
    For a $q$-matroid $M=(E,\rho)$, we say that $X\leq E$ is a \textbf{flat} if $\rho(X)<\rho(X\oplus\langle v\rangle)$ for all $v\in E\setminus X$. We say that $X$ is \textbf{cyclic} if $\rho(H)=\rho(X)$ for all $H\leq X$ such that $\dim(H)=\dim(X)-1$.
\end{definition}

\begin{notation}
    For a $q$-matroid $M$, we denote its set of flats by $\mathcal{F}(M)$ and its set of cyclic spaces by $\mathcal{C}(M)$. The set of cyclic flats is $\mathcal{F}(M)\cap\mathcal{C}(M)$, which we denote by $\mathcal{Z}(M)$. We may use the same established notation $\mathcal{I}$, $\mathcal{F}$, $\mathcal{C}$, and $\mathcal{Z}$ analogously for matroids.
\end{notation}

It was established in \cite{JurriusPellikaan2018,AlfaranoByrne2024CyclicFlats} that a $q$-matroid is uniquely determined not only by its rank function, but alternatively by its independent spaces, flats, cyclic spaces, or cyclic flats together with the rank function restricted to the cyclic flats. Therefore, a $q$-matroid may equivalently be defined using any one of these notions.

The following standard definitions from $q$-matroid theory will be convenient for our purposes.

\begin{definition}
    For a $q$-matroid $M=(E,\rho)$ and $X\leq E$, an inclusion-maximal element of $\mathcal{I}(M)$ that is contained in $X$ is called a \textbf{basis of $X$ in $M$}. An inclusion-minimal element of $\mathcal{C}(M)$ is called a \textbf{circuit of $M$}.
\end{definition}

It is well established that for any $X\leq E$, all bases of $X$ in $M$ are equidimensional. It was shown in \cite{JurriusPellikaan2018} that any $q$-matroid is uniquely determined by its set of independent spaces. It was shown in \cite{AlfaranoByrne2024CyclicFlats} that any $q$-matroid is uniquely determined by its set of cyclic flats together with the respective ranks of those cyclic flats.

One of the fundamental features of matroids and $q$-matroids is the presence of natural binary operations. In the $q$-analogue, such operations take place over vector spaces rather than sets, and we therefore introduce the following notation. Throughout, all vector spaces are assumed to be finite-dimensional over a finite field.

\begin{notation}
Let $E_1,\dots,E_m$ be vector spaces of dimensions $n_1,\dots,n_m$, respectively. We write
\[
E_1 \oplus \cdots \oplus E_m
\]
for the $(n_1+\cdots+n_m)$-dimensional vector space obtained by embedding $E_i$ in the coordinate block corresponding to the $i$-th summand. We identify each $E_i$ with its canonical image in this direct sum. When no confusion can arise, the zero subspace of any vector space is denoted simply by $0$.

For each $i \in \{1,\dots,m\}$ and each subspace $X \le E_1 \oplus \cdots \oplus E_m$, we denote by $\pi_i(X)$ the projection of $X$ onto $E_i$. If $X$ can be written in the form
\[
X=\bigoplus_{i=1}^m
(0\oplus\cdots\oplus0\oplus X_i\oplus0\oplus\cdots\oplus0),
\qquad X_i \le E_i,
\]
we simply write $X=\bigoplus_{i=1}^m X_i$.
\end{notation}

We illustrate this notation with a simple example.

\begin{example}
Let $E_1=E_2=E_3=\mathbb{F}_2^2$, each with canonical basis vectors $e_1,e_2$. The space
$E_1\oplus E_2\oplus E_3$
is naturally isomorphic to $\mathbb{F}_2^6$, whose canonical basis vectors we denote by $\bar e_1,\dots,\bar e_6$. With the above notation,
\[
0\oplus E_2\oplus E_3
=\langle \bar e_3,\bar e_4,\bar e_5,\bar e_6\rangle
=0\oplus\langle e_1,e_2\rangle\oplus\langle e_1,e_2\rangle.
\]

Let $X=\langle \bar e_1+\bar e_3+\bar e_6\rangle.$ Then
\[
\pi_1(X)=\langle e_1\rangle\le E_1,\qquad
\pi_2(X)=\langle e_1\rangle\le E_2,\qquad
\pi_3(X)=\langle e_2\rangle\le E_3.
\]
\end{example}

Two fundamental binary operations in $q$-matroid theory are the direct sum and free product, whose definitions were established in \cite{CeriaJurrius2024DirectSum} and \cite{AlfaranoByrneFulcher2025}, respectively. We restate their definitions here for convenience. In the context of binary operations, we assume that all vector spaces are over the same field.

\begin{definition}\label{def:direct_sum}\cite[Definition 38, Theorem 45]{CeriaJurrius2024DirectSum}
    Let $M_1=(E_1,\rho_1)$ and $M_2=(E_2,\rho_2)$ be $q$-matroids. The \textbf{direct sum} of $M_1$ and $M_2$ is the $q$-matroid $M_1\oplus M_2=(E_1\oplus E_2,\rho)$, where for each $X\leq E_1\oplus E_2$,
    $$\rho(X)=\min\{\rho_1(\pi_1(Y))+\rho_2(\pi_2(Y))+\dim(X)-\dim(Y):Y\leq X\}.$$
\end{definition}

\begin{definition}\label{def:free_product}\cite[Definition 20]{AlfaranoByrneFulcher2025}
    Let $M_1=(E_1,\rho_1)$ and $M_2=(E_2,\rho_2)$ be $q$-matroids. The \textbf{free product} of $M_1$ and $M_2$ is the $q$-matroid $M_1\sq M_2=(E_1\oplus E_2,\rho)$, where $X\in\mathcal{I}(M_1\sq M_2)$ if and only if
    $$\pi_1(X\cap(E_1\oplus0))\in\mathcal{I}(M_1)\qquad\text{and}\qquad
    \rho_1(E_1)-\rho_1(\pi_1(X\cap(E_1\oplus0)))\geq \dim(\pi_2(X))-\rho_2(\pi_2(X)).$$
\end{definition}

We recall the following result from \cite{GluesingLuerssenJany2024}, which draws an important parallel with matroid theory.

\begin{theorem}\label{thm:direct_sum_cyclic_flats}\cite[Theorem 6.2]{GluesingLuerssenJany2024}
    For $q$-matroids $M_1$ and $M_2$,
    $$\mathcal{Z}(M_1\oplus M_2)=\{Z_1\oplus Z_2:Z_1\in\mathcal{Z}(M_1)\textup{ and }Z_2\in\mathcal{Z}(M_2)\}.$$
\end{theorem}

The important parallel that the above theorem draws with matroid theory is that if $N_1$ and $N_2$ are matroids, then $\mathcal{Z}(N_1\oplus N_2)=\{Z_1\cup Z_2:Z_1\in\mathcal{Z}(N_1)\textup{ and }Z_2\in\mathcal{Z}(N_2)\}$. This analogy with respect to the direct sum does not occur for other $q$-matroidal notions such as independent spaces, flats, and so on.

In the recent paper \cite{saaltink2025theoryqtransversals}, the theory of $q$-transversals was developed. Moreover, a class of $q$-matroids arising from this theory was introduced. These objects are central to this paper, and we recall the relevant definitions here.

\begin{definition}
    Let $S$ be a finite set and let $\mathcal{A}$ be a collection of subsets of $S$. We say that $X\subseteq S$ has an \textbf{avoidance transversal of $\mathcal{A}$} if there is an ordering $X=\{x_1,\dots,x_\ell\}$ and an ordering $\mathcal{A}=(A_1,\dots,A_k)$ such that $x_i\notin A_i$ for each $i\in[\ell]$.
\end{definition}

From the above definition, it is clear that a necessary condition for $X$ to have an avoidance transversal of $\mathcal{A}$ is that $\ell\leq k$. It is a classical result that the set of subsets of $S$ that have avoidance transversals of $\mathcal{A}$ coincides with the set of independent sets of a matroid on the ambient set $S$. We recall the corresponding definition, which is analogous to the classical definition of a \emph{transversal matroid}.

\begin{definition}
    Let $\mathcal{A}=(A_1,\dots,A_k)$ be a collection of subsets of a finite set $S$. An \textbf{avoidance transversal matroid} $M=(S,r)$ with \textbf{presentation} $\mathcal{A}$ is the matroid whose independent sets are precisely those subsets $X\subseteq S$ that have an avoidance transversal of $\mathcal{A}$.
\end{definition}

By taking complements, it is clear that the class of transversal matroids coincides with the class of avoidance transversal matroids. For the $q$-analogue, however, it is necessary to use the notion of avoidance transversals.

\begin{definition}\cite{saaltink2025theoryqtransversals}
    Let $\mathcal{X}=(X_1,\dots,X_k)$ be a collection of subspaces of a vector space $E$. We say that $X\leq E$ is a \textbf{partial $q$-transversal of $\mathcal{X}$} if all linear bases of $X$ has an avoidance transversal of $\mathcal{X}$.
\end{definition}

From the above definition, it was shown in \cite{saaltink2025theoryqtransversals} that the set of partial $q$-transversals of $\mathcal{X}$ coincides with the set of independent spaces of a $q$-matroid. We therefore recall the following definition.

\begin{definition}\cite{saaltink2025theoryqtransversals}
    Let $\mathcal{X}=(X_1,\dots,X_k)$ be a collection of subspaces of a finite vector space $E$. A \textbf{transversal $q$-matroid} $M=(E,\rho)$ with \textbf{presentation} $\mathcal{X}$ is the $q$-matroid whose independent spaces are precisely the subspaces $X\leq E$ that are partial $q$-transversals of $\mathcal{X}$.
\end{definition}

In this paper, we leverage the following representability result from \cite{saaltink2025theoryqtransversals} to show the representability of a class of $q$-matroids not previously known to be representable.

\begin{theorem}\label{thm:saaltink_representable}\cite[Theorem 24]{saaltink2025theoryqtransversals}
    Let $\mathcal{X}$ be a collection of subspaces of $E$. If there exists a fixed linear basis $\beta$ of $E$ such that each member of $\mathcal{X}$ is the span of a subset of $\beta$, then the transversal $q$-matroid with presentation $\mathcal{X}$ is representable.
\end{theorem}

\section{The cyclic flat embedding theorem}

This section contains the main result of this paper. Our goal is to formalise a correspondence between matroids and a class of $q$-matroids whose cyclic flats are \emph{coordinate}. Through this correspondence, structural questions about such $q$-matroids can be reduced to the corresponding questions for matroids.

In particular, we show that if the cyclic flats of a $q$-matroid are coordinate, then the questions of whether it is transversal and representable can be decided by examining its \emph{corresponding matroid}. The resulting cyclic flat embedding theorem provides a mechanism for transferring structural results from transversal matroid theory to the $q$-matroid setting.

\begin{definition}
    A subspace of $\mathbb{F}_q^n$ spanned by a subset of the standard basis vectors is called \textbf{coordinate}. If $M$ is a $q$-matroid whose cyclic flats are all coordinate, we call $M$ a \textbf{coordinate $q$-matroid}.
\end{definition}

We fix the following notation.

\begin{definition}
    Let $\mathcal{L}(\mathbb{F}_q^n)$ denote the lattice of subspaces of $\mathbb{F}_q^n$. Define $\phi:2^{[n]}\to\mathcal{L}(\mathbb{F}_q^n)$ by 
    $$\phi(S)=\langle e_i:i\in S\rangle\quad \textup{for each } S\subseteq[n].$$ 
    For a subspace $X\leq\mathbb{F}_q^n$, we denote by $\phi^{-1}(X)$ the preimage of $X$ under $\phi$. We extend $\phi$ and $\phi^{-1}$ to subsets of the respective domains in the natural way.
\end{definition}

The following lemma establishes a correspondence between matroids and coordinate $q$-matroids. Its proof follows immediately from the cyclic flat axioms for matroids \cite[Theorem 3.2]{BoninDeMier2008} and $q$-matroids \cite[Definition 4.1]{AlfaranoByrne2024CyclicFlats}.

\begin{lemma}\label{lem:corresponding_matroid}
\begin{enumerate}
    \item For any coordinate $q$-matroid with weighted cyclic flats $(\mathcal{Z},\lambda)$, there is a matroid with weighted cyclic flats $(\phi^{-1}(\mathcal{Z}),\lambda\circ\phi)$.
    
    \item For any matroid with weighted cyclic flats $(\mathcal{Z},\lambda)$, there is a coordinate $q$-matroid with weighted cyclic flats $(\phi(\mathcal{Z}),\lambda\circ\phi^{-1})$.
\end{enumerate}
\end{lemma}

Using Lemma~\ref{lem:corresponding_matroid}, together with the fact that a $q$-matroid or matroid is uniquely determined by its weighted lattice of cyclic flats \cite{AlfaranoByrne2024CyclicFlats,BoninDeMier2008}, we make the following definition.

\begin{definition}
    Suppose that $M$ is a coordinate $q$-matroid with weighted cyclic flats $(\mathcal{Z},\lambda)$. We call the matroid with weighted cyclic flats $(\phi^{-1}(\mathcal{Z}),\lambda\circ\phi)$ the \textbf{corresponding matroid} of $M$.
    
    Suppose instead that $M$ is a matroid with weighted cyclic flats $(\mathcal{Z},\lambda)$. We call the $q$-matroid with weighted cyclic flats $(\phi(\mathcal{Z}),\lambda\circ\phi^{-1})$ the \textbf{corresponding $q$-matroid} of $M$.
\end{definition}

For two vectors $u,v\in\mathbb{F}_q^n$, we denote by $u\cdot v$ the standard dot product. In the following definition, we establish notation for what is often called the \emph{support} of a vector.

\begin{definition}
    Let $\sigma:\mathbb{F}_q^n\to 2^{[n]}$ be defined by $\sigma(v)=\{i\in[n]:v\cdot e_i\neq 0\}$ for all $v\in\mathbb{F}_q^n$. We extend $\sigma$ to subspaces of $\mathbb{F}_q^n$ by defining $\sigma(V)=\bigcup_{x\in V}\sigma(x)$ for each $V\leq \mathbb{F}_q^n$.
\end{definition}

We now focus on the correspondence between coordinate $q$-matroids and transversal matroids. We will show that if the corresponding matroid of a coordinate $q$-matroid is transversal, then this $q$-matroid is transversal and representable. As we will see, this correspondence provides a class of representable $q$-matroids that was not previously known to be representable.

For the remainder of this section, unless otherwise specified, we fix the following notation:
\begin{itemize}
    \item Let $M=([n],r)$ be an avoidance transversal matroid with presentation $\mathcal{S}=(S_1,\dots,S_k)$.
    \item Let $\phi(M)=(\mathbb{F}_q^n,\rho)$ be the transversal $q$-matroid with presentation $\mathcal{X}=(X_1,\dots,X_k)$, where $X_i=\phi(S_i)$ for each $i\in[k]$.
\end{itemize}

We will show that $\mathcal{Z}(\phi(M))=\phi(\mathcal{Z}(M))$ and that $\rho(\phi(Z))=r(Z)$ for all $Z\in\mathcal{Z}(M)$. Since a $q$-matroid is uniquely determined by its weighted lattice of cyclic flats, it will follow that the corresponding matroid of $\phi(M)$ is precisely $M$.

The following lemma will be used repeatedly. Its proof is immediate and therefore omitted.

\begin{lemma}\label{lem:vector_to_support_containment}
    For each $v\in\mathbb{F}_q^n$ and $i\in[k]$, we have $v\in X_i$ if and only if $\sigma(v)\subseteq S_i$.
\end{lemma}

We begin by showing that independent sets of $M$ map to independent spaces of $\phi(M)$.

\begin{lemma}\label{lem:indep_to_indep}
    Let $I\in\mathcal{I}(M)$. Then $\phi(I)\in\mathcal{I}(\phi(M))$.
\end{lemma}

\begin{proof}
    Let $\beta$ be a linear basis of $\phi(I)$. By Hall's theorem, the family $\{\sigma(b):b\in\beta\}$ admits a system of distinct representatives, giving a matching from $\beta$ to $I$. Since $I$ has an avoidance transversal in $\mathcal{S}$, Lemma~\ref{lem:vector_to_support_containment} implies that the matched coordinates yield an avoidance transversal of $\beta$ in $\mathcal{X}$. Since our choice of $\beta$ was arbitrary, it follows that $\phi(I)\in\mathcal{I}(\phi(M))$.
\end{proof}

Next, we show that the flats of $M$ are mapped to flats of $\phi(M)$.

\begin{lemma}\label{lem:flats_to_flats}
    Let $F\in\mathcal{F}(M)$. Then $\phi(F)\in\mathcal{F}(\phi(M))$.
\end{lemma}

\begin{proof}
    Let $B$ be a basis of $F$ in $M$. By Lemma~\ref{lem:indep_to_indep}, $\phi(B)\in\mathcal{I}(\phi(M))$. Let $v\in\mathbb{F}_q^n\setminus\phi(F)$. Then there exists $j\in\sigma(v)$ such that $j\notin F$, and hence $e_j\notin\phi(F)$.

    Let $\beta$ be any linear basis of $\phi(B)\oplus\langle v\rangle$. Since $j\in\sigma(v)$ and $j\notin F$, Hall's theorem implies that the family $\{\sigma(b):b\in\beta\}$ admits a matching covering $B\cup\{j\}$. By Lemma~\ref{lem:vector_to_support_containment}, $e_j\notin X_i$ implies $v\notin X_i$ for any $i\in[k]$.

    Since every linear basis of $\phi(B)\oplus\langle e_j\rangle$ has an avoidance transversal in $\mathcal{X}$, it follows that $\beta$ also has an avoidance transversal. Thus $\phi(B)\oplus\langle v\rangle\in\mathcal{I}(\phi(M))$.

    Therefore every $v\notin\phi(F)$ increases the rank of $\phi(F)$ in $\phi(M)$, and so $\phi(F)$ is a flat of $\phi(M)$.
\end{proof}

Next, we show that cycles of $M$ map to cycles of $\phi(M)$. For a subspace $V$, we denote by $\textup{Hyp}(V)$ the set of hyperplanes of $V$. We also establish the following notation.

\begin{notation}
    For $i\in[n]$, we denote by $\pi_{\setminus i}$ the canonical projection
    $$\pi_{\setminus i}:\mathbb{F}_q^n\to \langle e_1,\dots,e_{i-1},e_{i+1},\dots,e_n\rangle.$$
    We extend $\pi_{\setminus i}$ to subspaces in the natural way. Moreover, we denote
    $$\mathcal{X}_{\setminus i}=(\pi_{\setminus i}(X_1),\dots,\pi_{\setminus i}(X_k)).$$
\end{notation}

\begin{lemma}\label{lem:cycles_to_cycles}
    Let $C\in\mathcal{C}(M)$. Then $\phi(C)\in\mathcal{C}(\phi(M))$.
\end{lemma}

\begin{proof}
    We first prove the result in the case that $C$ is a circuit of $M$. Clearly, since there is no avoidance transversal of $C$, there is no avoidance transversal of $\{e_i:i\in C\}$ and therefore $\phi(C)$ is not independent in $\phi(M)$. For each $i\in C$, the set $C\setminus\{i\}$ is a basis of $C$. We will show that every $H<\phi(C)$ is independent in $\phi(M)$.
    
    Let $H_q\in\textup{Hyp}(\phi(C))$.
    By a basic result of linear algebra, there exists an $i\in C$ such that $\pi_{\setminus i}(H_q)=\phi(C\setminus\{i\})$. Therefore, for any linear basis $\beta$ of $H_q$, the set of vectors $\beta'=\{\pi_{\setminus i}(b):b\in\beta\}$ is a linear basis of $\phi(C\setminus\{i\})$. By Lemma~\ref{lem:indep_to_indep}, there is an avoidance transversal for $\beta'$ of $\mathcal{X}$. Since $\sigma(\pi_{\setminus i}(b))\subseteq\sigma(b)$ for all $b\in\beta$, it follows by Lemma~\ref{lem:vector_to_support_containment} that $\beta$ has an avoidance transversal of $\mathcal{X}$. It follows that $H_q$ is independent in $\phi(M)$. Therefore $\phi(C)$ is cyclic in $\phi(M)$.

    Now let $C\in\mathcal{C}(M)$ be arbitrary. Since the cycles of a matroid are precisely the unions of its circuits, we may write $C=\bigcup_{j=1}^m C_j$ where each $C_j$ is a circuit of $M$. By the first part of the proof, each $\phi(C_j)$ is cyclic in $\phi(M)$. Since
    \[
        \phi(C)=\phi\left(\bigcup_{j=1}^m C_j\right)=\sum_{j=1}^m \phi(C_j),
    \]
    and the sum of cyclic spaces in a $q$-matroid is cyclic, it follows that $\phi(C)\in\mathcal{C}(\phi(M))$.
\end{proof}

We now show that the only cyclic flats of $\phi(M)$ lying in the image of $\phi$ arise from cyclic flats of $M$. This follows from the following lemma.

\begin{lemma}
    Let $X\in 2^{[n]}\setminus\mathcal{F}(M)$. Then $\phi(X)\notin\mathcal{F}(\phi(M))$.
\end{lemma}

\begin{proof}
    Let $B$ be a basis of $X$ in $M$. Since $X$ is not a flat, there exists $j\in[n]\setminus X$ such that $B\cup\{j\}\notin\mathcal{I}(M)$. Thus there is no avoidance transversal of $B\cup\{j\}$ in $\mathcal{S}$. By Lemma~\ref{lem:vector_to_support_containment}, there is therefore no avoidance transversal of
    $\beta=\{e_b:b\in B\}\cup\{e_j\}$
    in $\mathcal{X}$.

    Since $\beta$ is a linear basis of $\phi(B)\oplus\langle e_j\rangle$, it follows that
    \[
        \rho(\phi(B)\oplus\langle e_j\rangle)
        < \dim(\phi(B)\oplus\langle e_j\rangle)
        = \dim(\phi(B))+1.
    \]
    Hence $\rho(\phi(B)\oplus\langle e_j\rangle)=\rho(\phi(B))$. By submodularity, this implies that
    $\rho(\phi(X))=\rho(\phi(X)\oplus\langle e_j\rangle)$,
    and therefore $\phi(X)\notin\mathcal{F}(\phi(M))$.
\end{proof}

The following corollary now follows immediately from the preceding lemmas.

\begin{corollary}\label{cor:CFs_in_image_are_image_of_CFs}
    The cyclic flats of $\phi(M)$ that lie in the image of $\phi$ are precisely $\phi(\mathcal{Z}(M))$.
\end{corollary}

In the following lemma, we show that the cyclic flats of $\phi(M)$ are contained in the image of $\phi$ (that is, $\phi(M)$ is coordinate).

\begin{lemma}\label{lem:CFs_of_q-mat_in_image}
    All cyclic flats of $\phi(M)$ lie in the image of $\phi$.
\end{lemma}

\begin{proof}
    We proceed by induction on $n$. The base case $n\leq 2$ holds since all such matroids and $q$-matroids are respectively transversal and transversal $q$-matroids, and there is a one-to-one correspondence between the respective equivalence classes. Suppose that the statement holds for all ambient dimensions at most $k$.

    Let $n=k+1$ and let $Z\in\mathcal{Z}(\phi(M))$. Suppose that $i\in\sigma(Z)$ but $e_i\notin Z$. Let
    \[
        H=\langle e_1,\dots,e_{i-1},e_{i+1},\dots,e_n\rangle .
    \]
    Since $i\in\sigma(Z)$, we have $Z\not\leq H$, and therefore $H\cap Z\in\textup{Hyp}(Z)$. Moreover, since $i\in\sigma(Z)$ but $e_i\notin Z$, the map $\pi_{\setminus i}$ is injective on $Z$. Also observe that $\pi_{\setminus i}$ is the identity on $H$.

    We now show that $\pi_{\setminus i}(Z)$ is a cyclic flat of the restricted $q$-matroid $\phi(M)|H$. This will allow us to apply the induction hypothesis to $\pi_{\setminus i}(Z)$.

    First we show that $\pi_{\setminus i}(Z)$ is a flat in $\phi(M)|H$. Since $Z$ is cyclic, every hyperplane of $Z$ contains a basis of $Z$. Hence there exists a basis $B$ of $Z$ such that $B\leq H$. Since $Z$ is a flat, we have $B\oplus\langle v\rangle\in\mathcal{I}(\phi(M))$ for all $v\notin Z$. For $v\in H\setminus Z$, we have $i\notin\sigma(B)$ and $i\notin\sigma(v)$, and therefore $i\notin\sigma(B\oplus\langle v\rangle)$. By Lemma~\ref{lem:vector_to_support_containment}, it follows that $B\oplus\langle v\rangle$ is a partial $q$-transversal of $\mathcal{X}_{\setminus i}$. Hence for every $v\in H\setminus\pi_{\setminus i}(Z)$, the rank of $\pi_{\setminus i}(Z)\oplus\langle v\rangle$ in $\phi(M)|H$ strictly exceeds that of $\pi_{\setminus i}(Z)$. Thus $\pi_{\setminus i}(Z)$ is a flat in $\phi(M)|H$.

    Next we show that $\pi_{\setminus i}(Z)$ is cyclic in $\phi(M)|H$. Let $B'$ be an arbitrary basis of $Z$. Since $Z$ is a flat and $e_i\notin Z$, we have $B'\oplus\langle e_i\rangle\in\mathcal{I}(\phi(M))$. Hence every linear basis of $B'\oplus\langle e_i\rangle$ has an avoidance transversal in $\mathcal{X}$. In particular, every such basis of the form $\beta\cup\{e_i\}$ with $i\notin\sigma(\beta)$ has an avoidance transversal in $\mathcal{X}$. By Lemma~\ref{lem:vector_to_support_containment}, we deduce that $\pi_{\setminus i}(\beta)$ has an avoidance transversal in $\mathcal{X}_{\setminus i}$. It follows that $\pi_{\setminus i}(B')$ is a partial $q$-transversal of $\mathcal{X}_{\setminus i}$.

    Therefore for any basis $B'$ of $Z$, the subspace $\pi_{\setminus i}(B')$ is a basis of $\pi_{\setminus i}(Z)$ in $\phi(M)|H$. Since $\pi_{\setminus i}$ is injective on $Z$ and is the identity on $H\cap Z$, it follows that for each $G\in\textup{Hyp}(\pi_{\setminus i}(Z))$ there exists a basis of $\pi_{\setminus i}(Z)$ contained in $G$. Hence $\pi_{\setminus i}(Z)$ is cyclic in $\phi(M)|H$. Since we already showed that $\pi_{\setminus i}(Z)$ is a flat, it follows that $\pi_{\setminus i}(Z)$ is a cyclic flat of $\phi(M)|H$.

    By the induction hypothesis, $\pi_{\setminus i}(Z)$ has a linear basis $\{e_j:j\in S\}$ for some $S\subseteq[n]\setminus\{i\}$. Since $\pi_{\setminus i}$ is the identity on $H$, it follows that $H\cap Z$ has a linear basis $\beta=\{e_j:j\in S\setminus\{\ell\}\}$ for some $\ell\in S$. Consequently $Z$ has a linear basis of the form $\beta\cup\{e_\ell+\lambda e_i\}$ for some $\lambda\in\mathbb{F}_q$.

    Since $Z$ is a flat, submodularity of $\rho$ implies that for all $v\notin Z$,
    \[
        \rho((H\cap Z)\oplus\langle v\rangle)-\rho(H\cap Z)
        \geq
        \rho(Z\oplus\langle v\rangle)-\rho(Z)
        =1 .
    \]
    If $\lambda\neq 0$, then this implies that
    \[
        \rho((H\cap Z)\oplus\langle e_k\rangle)>\rho(H\cap Z)
    \]
    for all $k\notin S\setminus\{\ell\}$. Therefore $H\cap Z$ is the image of a flat in $M$. By Lemma~\ref{lem:flats_to_flats}, we conclude that $H\cap Z$ is a flat in $\phi(M)$.

    However this implies the contradiction
    \[
        \rho(H\cap Z)<\rho(Z)=\rho(H\cap Z),
    \]
    since $Z$ is cyclic. Hence $\lambda=0$, which implies that $i\notin\sigma(Z)$, contradicting the original assumption.

    We therefore conclude that if $e_i\notin Z$, then $i\notin\sigma(Z)$. The result follows.
\end{proof}

Combining the above lemmas, we have shown that $\phi(\mathcal{Z}(M))=\mathcal{Z}(\phi(M))$. We now show that $\rho(\phi(Z))=r(Z)$ for each $Z\in\mathcal{Z}(M)$. We will then be ready to prove the main theorem of this paper.

\begin{lemma}\label{lem:ranks_agree}
    For all $Z\in\mathcal{Z}(M)$, we have $\rho(\phi(Z))=r(Z)$.
\end{lemma}

\begin{proof}
    Let $Z\in\mathcal{Z}(M)$ and let $B$ be a basis of $Z$ in $M$. By Lemma~\ref{lem:indep_to_indep}, $\phi(B)\in\mathcal{I}(\phi(M))$. We claim that $\phi(B)$ is a basis of $\phi(Z)$ in $\phi(M)$.

    Suppose that $\phi(B)\oplus\langle v\rangle\in\mathcal{I}(\phi(M))$ for some $v\in \phi(Z)\setminus \phi(B)$. Let $\beta$ be a linear basis of $\phi(B)\oplus\langle v\rangle$ containing $v$. Since $v\notin\phi(B)$, there exists $\ell\in\sigma(v)\setminus B$. As $v\in\phi(Z)$, we have $\ell\in Z$.

    Since $\phi(B)\oplus\langle v\rangle$ is independent, the basis $\beta$ has an avoidance transversal in $\mathcal{X}$. By Lemma~\ref{lem:vector_to_support_containment}, this implies that the set $\{e_i:i\in B\}\cup\{e_\ell\}$ has an avoidance transversal in $\mathcal{X}$. It follows that $B\cup\{\ell\}$ has an avoidance transversal in $\mathcal{S}$, and hence $B\cup\{\ell\}\in\mathcal{I}(M)$. This contradicts the fact that $B$ is a basis of $Z$.

    Therefore $\phi(B)$ is a basis of $\phi(Z)$ in $\phi(M)$, and we conclude that
    $\rho(\phi(Z))=\dim(\phi(B))=|B|=r(Z)$.
\end{proof}

The following theorem is the main result of this paper. For its proof, we no longer assume the previously fixed notation $M=([n],r)$ and $\phi(M)=(\mathbb{F}_q^n,\rho)$.

\begin{theorem}[The cyclic flat embedding theorem for transversal matroids]\label{thm:cyc_flat_embedding}
    If the corresponding matroid of a coordinate $q$-matroid is transversal, then the $q$-matroid is transversal and representable.
\end{theorem}

\begin{proof}
    Let $M_q$ be a coordinate $q$-matroid whose corresponding matroid is $M=([n],r)$. Suppose that $M$ is avoidance transversal with presentation $\mathcal{S}=(S_1,\dots,S_k)$. Let $\phi(M)=(\mathbb{F}_q^n,\rho)$ be the transversal $q$-matroid with presentation $\mathcal{X}=(X_1,\dots,X_k)$, where $X_i=\phi(S_i)$ for each $i\in[k]$.

    By Corollary~\ref{cor:CFs_in_image_are_image_of_CFs}, Lemma~\ref{lem:CFs_of_q-mat_in_image}, and Lemma~\ref{lem:ranks_agree}, we obtain that $\mathcal{Z}(\phi(M))=\phi(\mathcal{Z}(M))$ and that $r(Z)=\rho(\phi(Z))$ for all $Z\in\mathcal{Z}(M)$.

    By the definition of a corresponding matroid, we have $\mathcal{Z}(M_q)=\phi(\mathcal{Z}(M))$, and for each $Z\in\mathcal{Z}(M_q)$ its rank in $M_q$ is $r(\phi^{-1}(Z))$. Observe that
    \[
        r(\phi^{-1}(Z))
        =
        \rho(\phi(\phi^{-1}(Z)))
        =
        \rho(Z).
    \]

    It follows that $\mathcal{Z}(M_q)=\mathcal{Z}(\phi(M))$ and that the ranks agree on these cyclic flats. By \cite[Theorem 4.11]{AlfaranoByrne2024CyclicFlats}, the $q$-matroids $M_q$ and $\phi(M)$ are equal. By \cite[Theorem 24]{saaltink2025theoryqtransversals}, $\phi(M)$ is a representable transversal $q$-matroid. The result follows.
\end{proof}

Theorem~\ref{thm:cyc_flat_embedding} allows us to apply tools from transversal matroid theory to representable $q$-matroid theory from the perspective of cyclic flats. We highlight one instance of this in the following corollary, in which we identify a class of representable $q$-matroids not previously known to be representable.

\begin{corollary}
    All nested $q$-matroids are transversal and representable.
\end{corollary}

\begin{proof}
    Let $M_q=(\mathbb{F}_q^n,\rho)$ be a nested $q$-matroid. This means that its lattice of cyclic flats is a chain. Therefore there exists a linear basis $\beta$ of $\mathbb{F}_q^n$ such that each cyclic flat is the span of a subset of $\beta$. Hence, up to a change of basis (and thus $q$-matroid equivalence), we may assume that $\beta$ is the set of standard basis vectors of $\mathbb{F}_q^n$. It follows that $M_q$ is coordinate and that its corresponding matroid is a nested matroid, which is known to be transversal. By Theorem~\ref{thm:cyc_flat_embedding}, it follows that $M_q$ is transversal and representable.
\end{proof}

A benefit of working with cyclic flats is that they form a structural feature shared by both matroids and $q$-matroids. This naturally raises the question of whether similar embedding results hold for other classes of matroids. That is, if a matroid $M$ has a property $P$, does it follow that its corresponding $q$-matroid also has property $P$? The following conjecture provides a plausible example.

\begin{conjecture}
    If the corresponding matroid of a coordinate $q$-matroid is representable, then the $q$-matroid is representable.
\end{conjecture}

\section{Closure of \texorpdfstring{$q$}{q}-transversal matroids under binary operations}

In this section we investigate the behaviour of transversal $q$-matroids under two fundamental binary operations: the direct sum and the free product. Our aim is to show that the class of transversal $q$-matroids is closed under these operations.

As shown in \cite{AlfaranoByrneFulcher2025} and \cite{GluesingLuerssenJany2024}, the structures of the cyclic flats of the free product and direct sum of $q$-matroids behave in a manner analogous to the matroid setting. In particular, for two matroids $M_1$ and $M_2$ we have
\[
\mathcal{Z}(\phi(M_1\oplus M_2))=\mathcal{Z}(\phi(M_1)\oplus\phi(M_2))
\quad\text{and}\quad
\mathcal{Z}(\phi(M_1\sq M_2))=\mathcal{Z}(\phi(M_1)\sq\phi(M_2)).
\]

Consequently, if every transversal $q$-matroid were coordinate up to equivalence, then by Theorem~\ref{thm:cyc_flat_embedding} the closure of transversal $q$-matroids under these binary operations would follow directly from the corresponding closure of transversal matroids. We therefore begin by verifying that this is not the case.

The following example shows that there exist transversal $q$-matroids that are not coordinate up to equivalence. It is the smallest such example that we were able to find.

\begin{example}
Let $M=(\mathbb{F}_q^6,\rho)$ be the transversal $q$-matroid with presentation
\[
\mathcal{X}=(X_1,X_2,X_3)
=
(\langle e_1,e_2,e_3\rangle,\langle e_4,e_5,e_6\rangle,\langle e_1+e_4,e_2+e_5,e_3+e_6\rangle).
\]

We show that $X_1$ is cyclic in $M$. Since $\dim(X_1)=3$ and $X_1$ intersects both $X_2$ and $X_3$ trivially, each hyperplane of $X_1$ is a partial $q$-transversal. Moreover, $X_1$ itself is not a transversal of $\mathcal{X}$. Therefore $X_1$ is cyclic.

By the same argument, $X_2$ and $X_3$ are also cyclic. By \cite[Theorem 21]{saaltink2025theoryqtransversals}, it follows that $\{X_1,X_2,X_3\}\subseteq\mathcal{Z}(M)$, and hence $M$ is not equivalent to a coordinate $q$-matroid.
\end{example}

We begin with the following fundamental lemma.

\begin{lemma}\label{lem:lin_basis_basis}
Let $M=(E,\rho)$ be a $q$-matroid. Any linear basis of $E$ contains a linear basis of a basis of $M$.
\end{lemma}

\begin{proof}
Let $\beta=\{b_1,\dots,b_n\}$ be a linear basis of $E$, and suppose that $B=\langle e_1,\dots,e_k\rangle$ is a maximally independent space in $M$ spanned by elements of $\beta$. Then for each $i\in[n]\setminus[k]$ we have
\[
\rho(B)=\rho(B\oplus\langle b_i\rangle).
\]
By submodularity it follows that $\rho(B)=\rho(E)$. Hence $B$ is a basis of $M$, and $\{b_1,\dots,b_k\}$ is a linear basis of $B$ contained in $\beta$.
\end{proof}

For the remainder of this section we fix the following notation.

\begin{notation}
Let $M_1=(E_1,\rho_1)$ and $M_2=(E_2,\rho_2)$ be transversal $q$-matroids with respective presentations
\[
\mathcal{X}_1=(A_1,\dots,A_k)
\quad\text{and}\quad
\mathcal{X}_2=(B_1,\dots,B_\ell).
\]

Define
\[
\mathcal{X}_1\sq\mathcal{X}_2
=
(A_1\oplus 0,\dots,A_k\oplus 0,E_1\oplus B_1,\dots,E_1\oplus B_\ell)
\]
and
\[
\mathcal{X}_1\oplus\mathcal{X}_2
=
(A_1\oplus E_2,\dots,A_k\oplus E_2,E_1\oplus B_1,\dots,E_1\oplus B_\ell).
\]
\end{notation}

We now establish that the class of transversal $q$-matroids is closed under free products.

\begin{theorem}\label{thm:free_product_transversal}
The $q$-matroid $M_1\sq M_2$ is a transversal $q$-matroid with presentation $\mathcal{X}_1\sq\mathcal{X}_2$.
\end{theorem}

\begin{proof}
Let $M$ be the transversal $q$-matroid with presentation $\mathcal{X}_1\sq\mathcal{X}_2$. We show that $\mathcal{I}(M)=\mathcal{I}(M_1\sq M_2)$. Let $X\le E_1\oplus E_2$.

It is straightforward to verify that $M|(E_1\oplus0)\equiv M_1$ under the isomorphism $\pi_1|_{(E_1\oplus0)}$. Therefore $\pi_1(X\cap(E_1\oplus0))\in\mathcal{I}(M_1)$ is a necessary condition for $X\in\mathcal{I}(M)$, which coincides with the corresponding necessary condition for $X\in\mathcal{I}(M_1\sq M_2)$.

Given this observation, we show that $X$ is a transversal of $\mathcal{X}_1\sq\mathcal{X}_2$ if and only if
\[
k-\dim(X\cap(E_1\oplus0))
\ge
\dim(\pi_2(X))-\rho_2(\pi_2(X)).
\]

First observe that for any $v\in E_1\oplus E_2$ and $i\in[\ell]$ we have
\[
v\notin E_1\oplus B_i
\quad\text{if and only if}\quad
\pi_2(v)\notin B_i.
\]

Let $\beta$ be a linear basis of $X$. If $\beta$ admits an avoidance transversal of $\mathcal{X}_1\sq\mathcal{X}_2$, then any linear basis $\tilde{\beta}$ of $X$ satisfying
\[
\tilde{\beta}\cap(E_1\oplus0)\subseteq\beta\cap(E_1\oplus0)
\quad\text{and}\quad
\beta\setminus(E_1\oplus0)\subseteq\tilde{\beta}\setminus(E_1\oplus0)
\]
also admits such a transversal. Hence we may assume
\[
|\beta\cap(E_1\oplus0)|=\dim(X\cap(E_1\oplus0)),
\]
which implies
\[
|\beta\setminus(E_1\oplus0)|=|\pi_2(\beta)|=\dim(\pi_2(X)).
\]

Let $\beta_1=\beta\cap(E_1\oplus0)$. Since $\pi_1(X\cap(E_1\oplus0))\in\mathcal{I}(M_1)$, we may fix an avoidance transversal for $\beta_1$ using $|\beta_1|=\dim(X\cap(E_1\oplus0))$ members of $(A_1\oplus0,\dots,A_k\oplus0)$.
Since $\langle\pi_2(\beta)\rangle=\pi_2(X)$, Lemma~\ref{lem:lin_basis_basis} implies that there exists a linearly independent set $\beta'\subseteq\pi_2(\beta)$ such that $\langle\beta'\rangle$ is a basis of $\pi_2(X)$ in $M_2$. In particular $|\beta'|=\rho_2(\pi_2(X))$.
Choose $\beta_2\subseteq\beta\setminus(E_1\oplus0)$ such that $\pi_2(\beta_2)=\beta'$ and $|\beta_2|=|\beta'|$. The avoidance transversal of $\beta'$ in $\mathcal{X}_2$ induces an avoidance transversal of $\beta_2$ in $(E_1\oplus B_1,\dots,E_1\oplus B_\ell)$.

The remaining vectors of $\beta$ are
\[
\beta\setminus(\beta_1\cup\beta_2),
\]
whose number equals
\[
\dim(\pi_2(X))-\rho_2(\pi_2(X)).
\]
Thus there are enough remaining members of $(A_1\oplus0,\dots,A_k\oplus0)$ to complete an avoidance transversal of $\beta$ if and only if
\[
\dim(\pi_2(X))-\rho_2(\pi_2(X))
\le
k-\dim(X\cap(E_1\oplus0)).
\]
This shows that $X\in\mathcal{I}(M)$ if and only if $X\in\mathcal{I}(M_1\sq M_2)$, and hence $\mathcal{I}(M)=\mathcal{I}(M_1\sq M_2)$.
\end{proof}

By \cite[Theorem 30]{AlfaranoByrneFulcher2025}, it is clear that if $M_1$ and $M_2$ are coordinate $q$-matroids, then $M_1\sq M_2$ is coordinate. Therefore we immediately obtain the following corollary.

\begin{corollary}\label{cor:free_prod}
    Let $M_1$ and $M_2$ be coordinate transversal $q$-matroids. Then $M_1\sq M_2$ is a coordinate transversal $q$-matroid.
\end{corollary}

We now turn to the case of the direct sum. This situation is more delicate than the free product, so we restrict ourselves here to the direct sum of coordinate transversal $q$-matroids. Nevertheless, the result suggests a natural presentation for the direct sum in the general case.

\begin{lemma}\label{lem:direct_sum_matroids}
Let $N_1=(S,r_1)$ and $N_2=(T,r_2)$ be avoidance transversal matroids with respective presentations $(S_1,\dots,S_k)$ and $(T_1,\dots,T_\ell)$, where $S\cap T=\varnothing$. Then $N_1\oplus N_2$ is an avoidance transversal matroid with presentation
\[
\mathcal{S}=(S_1\cup T,\dots,S_k\cup T,S\cup T_1,\dots,S\cup T_\ell).
\]
\end{lemma}

\begin{proof}
Let $N$ denote the avoidance transversal matroid with presentation $\mathcal{S}$. We show that $N=N_1\oplus N_2$.

Observe that no element of $S$ avoids any member of $(S\cup T_1,\dots,S\cup T_\ell)$, and no element of $T$ avoids any member of $(S_1\cup T,\dots,S_k\cup T)$. Hence a set $I\subseteq S\cup T$ is an avoidance transversal of $\mathcal{S}$ if and only if $I\cap S$ is an avoidance transversal of $(S_1,\dots,S_k)$ and $I\cap T$ is an avoidance transversal of $(T_1,\dots,T_\ell)$.

Therefore $I$ is independent in $N$ if and only if $I\cap S$ is independent in $N_1$ and $I\cap T$ is independent in $N_2$. This is exactly the definition of independence in $N_1\oplus N_2$, and hence $N=N_1\oplus N_2$.
\end{proof}

\begin{theorem}\label{thm:coord_q-mat_direct_sum}
If $M_1$ and $M_2$ are coordinate transversal $q$-matroids, then $M_1\oplus M_2$ is a coordinate transversal $q$-matroid with presentation $\mathcal{X}_1\oplus\mathcal{X}_2$.
\end{theorem}

\begin{proof}
Since $M_1$ and $M_2$ are coordinate, we may fix linear bases
$\beta_1=\{a_1,\dots,a_m\}$ and $\beta_2=\{b_1,\dots,b_n\}$ of $E_1$ and $E_2$
such that every member of $\mathcal{Z}(M_1)$ and $\mathcal{Z}(M_2)$ is the span of a subset of $\beta_1$ and $\beta_2$. Let $\phi^{-1}(M_1)$ and $\phi^{-1}(M_2)$ denote the corresponding matroids.

By \cite[Theorems 21 and 23]{saaltink2025theoryqtransversals}, we have $A_i\in\mathcal{Z}(M_1)$ and $B_j\in\mathcal{Z}(M_2)$ for each $i\in[k]$ and $j\in[\ell]$, so every member of the presentations $\mathcal{X}_1$ and $\mathcal{X}_2$ is coordinate. Consequently, by Theorem~\ref{thm:cyc_flat_embedding}, the families
\[
\mathcal{S}_1=(\sigma(A_1),\dots,\sigma(A_k))
\quad\text{and}\quad
\mathcal{S}_2=(\sigma(B_1),\dots,\sigma(B_\ell))
\]
are presentations of $\phi^{-1}(M_1)$ and $\phi^{-1}(M_2)$.

We regard $\phi^{-1}(M_1)$ and $\phi^{-1}(M_2)$ as matroids on disjoint ground sets $S$ and $T$. By Lemma~\ref{lem:direct_sum_matroids}, the direct sum $\phi^{-1}(M_1)\oplus\phi^{-1}(M_2)$ is a transversal matroid with presentation
\[
(\sigma(A_1)\cup T,\dots,\sigma(A_k)\cup T,S\cup\sigma(B_1),\dots,S\cup\sigma(B_\ell)).
\]

By \cite[Theorem 6.2]{GluesingLuerssenJany2024}, the $q$-matroid associated to $\phi^{-1}(M_1)\oplus\phi^{-1}(M_2)$ is precisely $M_1\oplus M_2$. Applying Theorem~\ref{thm:cyc_flat_embedding}, it follows that $M_1\oplus M_2$ has a presentation obtained by applying $\phi$ to the members of the above family. This yields
\[
(A_1\oplus E_2,\dots,A_k\oplus E_2,E_1\oplus B_1,\dots,E_1\oplus B_\ell),
\]
which is exactly $\mathcal{X}_1\oplus\mathcal{X}_2$.
\end{proof}

Theorem~\ref{thm:coord_q-mat_direct_sum} suggests that if the class of transversal $q$-matroids is closed under direct sums, then the presentation of any such direct sum can be formed analogously to the matroid case. We therefore record the following conjecture.

\begin{conjecture}
Let $M_1$ and $M_2$ be transversal $q$-matroids with presentations $\mathcal{X}_1$ and $\mathcal{X}_2$. Then $M_1\oplus M_2$ is a transversal $q$-matroid with presentation $\mathcal{X}_1\oplus\mathcal{X}_2$.
\end{conjecture}

The following consequence illustrates how the preceding results recover and extend known representability results for uniform $q$-matroids. Its proof follows immediately from Corollary~\ref{cor:free_prod} and Theorem~\ref{thm:coord_q-mat_direct_sum}.

\begin{corollary}
Let $M_1,\dots,M_t$ be uniform $q$-matroids. Then any $q$-matroid obtained from $M_1,\dots,M_t$ by iteratively applying direct sums and free products is a coordinate transversal $q$-matroid and therefore representable.
\end{corollary}

In particular, this recovers the representability of arbitrary direct sums of uniform $q$-matroids established in \cite{alfarano2026representability}, which may suggest a useful connection between $q$-transversals and finite geometry.

\bibliography{references}
\bibliographystyle{abbrv}
\end{document}